\documentclass[a4paper,12pt]{article}
\usepackage{latexsym}
\usepackage{amsmath}
\usepackage{ams	thm}
\usepackage[dvipdfmx]{graphicx}
\usepackage{txfonts}
\usepackage{bm}
\usepackage{color}
\usepackage{epic,eepic}

\usepackage{geometry}
\numberwithin{equation}{section}

\newtheorem{theorem}{Theorem}[section]
\newtheorem{proposition}{Proposition}[section]
\newtheorem{corollary}{Corollary}[section]
\newtheorem{lemma}{Lemma}[section]
\newtheorem{remark}{Remark}[section]
\newtheorem{example}{Example}[section]

\newcommand{\OMIT}[1]{{\bf [OMIT:} #1 \ {\bf --- end OMIT] }}  
   \renewcommand{\OMIT}[1]{}            

\newcommand{\RR}{{\mathbb{R}}}
\newcommand{\ZZ}{{\mathbb{Z}}}

\newcommand{\veczero}{{\bf 0}}
\newcommand{\dom}{{\rm dom\,}}

\newcommand{\domZ}{{\rm dom_{\ZZ}}}
\newcommand{\calF}{\mathcal{F}}

\newcommand{\subgR}{\partial_{\RR}}
\newcommand{\subgZ}{\partial_{\ZZ}}
\newcommand{\Lnat}{{L$^{\natural}$}}
\newcommand{\Mnat}{{M$^{\natural}$}}

\newcommand{\kwd}[1]{{\em #1\/}}  
\newcommand{\finbox}{\hspace*{\fill}$\rule{0.2cm}{0.2cm}$}
\newcommand{\todaye}{\the\year/\the\month/\the\day}

\begin{document}

\title{Integrality of Subgradients and Biconjugates of \\
Integrally Convex Functions%
}

\author{
Kazuo Murota%
\thanks{Department of Economics and Business Administration,
Tokyo Metropolitan University, 
murota@tmu.ac.jp}
\ and 
Akihisa Tamura%
\thanks{Department of Mathematics, Keio University, 
aki-tamura@math.keio.ac.jp}
}

\date{June 2018; September 2018}

\maketitle

\begin{abstract}
Integrally convex functions constitute a fundamental function class
in discrete convex analysis.
This paper shows that an integer-valued integrally convex function
admits an integral subgradient and that the integral biconjugate 
of an integer-valued integrally convex function coincides with itself.
The proof is based on the Fourier--Motzkin elimination.
The latter result provides a unified proof of integral biconjugacy
for various classes of integer-valued discrete convex functions, 
including L-convex, M-convex, L$_{2}$-convex, M$_{2}$-convex,
BS-convex, and UJ-convex functions as well as multimodular functions.
Our results of integral subdifferentiability and integral biconjugacy
make it possible to extend the theory of discrete DC (difference of convex) functions 
developed for L- and M-convex functions to that for integrally convex functions,
including an analogue of the Toland--Singer duality for integrally convex functions.
\end{abstract}



\section{Introduction}
\label{SCintro}

In discrete convex analysis \cite{Fuj05book,Mdca98,Mdcasiam},
a variety of discrete convex functions are considered.
Among others, integrally convex functions,
due to Favati--Tardella \cite{FT90},
constitute a common framework for discrete convex functions,
and almost all kinds of discrete convex functions 
are known to be integrally convex. 
Indeed, separable convex,
{\rm L}-convex, ${\rm L}^{\natural}$-convex, {\rm M}-convex,  
${\rm M}^{\natural}$-convex,  ${\rm L}^{\natural}_{2}$-convex, and 
${\rm M}^{\natural}_{2}$-convex functions are known to be integrally convex \cite{Mdcasiam}.
Multimodular functions \cite{Haj85} 
are also integrally convex, 
as pointed out in \cite{Mdcaprimer07}.
Moreover, BS-convex and UJ-convex functions \cite{Fuj14bisubmdc}
are integrally convex.

The concept of integral convexity
is used in formulating discrete fixed point theorems
and found applications in economics and game theory
\cite{IMT05,Mdcaeco16,Yan09fixpt}.
A proximity theorem for integrally convex functions
has recently been established in \cite{MMTT17proxIC} 
together with a proximity-scaling algorithm for minimization.
Fundamental operations for integrally convex functions
such as projection and convolution are investigated in \cite{MM17projcnvl}.

In this paper we are concerned with subgradients and biconjugates
of integer-valued integrally convex functions.
For a function $f: \ZZ\sp{n} \to \RR \cup \{ +\infty \}$
we denote its effective domain as 
$\domZ f = \{ x \in \ZZ\sp{n} \mid f(x) < +\infty \}$;
we always assume that $\domZ f$ is nonempty.
For an integer-valued function 
$f: \ZZ\sp{n} \to \ZZ \cup \{ +\infty \}$,
we define
$f\sp{\bullet}: \ZZ\sp{n} \to \ZZ \cup \{ +\infty \}$ 
by
\begin{align}
f\sp{\bullet}(p)  &= \sup\{  \langle p, x \rangle - f(x)   \mid x \in \ZZ\sp{n} \}
\qquad ( p \in \ZZ\sp{n}),
 \label{conjvexZpZ} 
\end{align}
where 
$\langle p, x \rangle = \sum_{i=1}\sp{n} p_{i} x_{i}$
is the inner product of 
$p=(p_{1}, p_{2}, \ldots, p_{n})$ and 
$x=(x_{1}, x_{2}, \allowbreak \ldots, \allowbreak  x_{n})$.
This function $f\sp{\bullet}$ is referred to as the  
\kwd{integral conjugate} of $f$.
We can apply (\ref{conjvexZpZ}) twice to obtain
$f\sp{\bullet\bullet} = (f\sp{\bullet})\sp{\bullet}$,
which is called the \kwd{integral biconjugate} of $f$.

Concerning conjugacy and biconjugacy 
it is natural to ask the following questions
for a given class of discrete convex functions.

\begin{itemize}
\item
For an integer-valued function $f$ in the class,
does the integral conjugate $f\sp{\bullet}$
belong to the same class?
If not, how is it characterized?

\item
For an integer-valued function $f$ in the class,
does integral biconjugacy $f\sp{\bullet\bullet} =f$ hold?
\end{itemize}
These questions are completely settled for
separable convex,
{\rm L}-convex,
${\rm L}^{\natural}$-convex,
{\rm M}-convex,  
${\rm M}^{\natural}$-convex,  
${\rm L}^{\natural}_{2}$-convex, and 
${\rm M}^{\natural}_{2}$-convex functions;
see \cite[Chapter 8]{Mdcasiam}.
We may say that they are also settled for multimodular functions
via equivalence between ${\rm L}^{\natural}$-convexity and multimodularity
pointed out in \cite{Mmult05}.
The conjugacy question for BS-convex and UJ-convex functions
is addressed in \cite{Fuj14bisubmdc}.

For integrally convex functions,
the first question about conjugacy is already settled in the negative \cite{MS01rel}.
Indeed, there is an example of an integrally convex function
whose integral conjugate is not integrally convex;
see Remark \ref{RMconjIC} in Section \ref{SCintcnvfn}.
The main result of this paper is an affirmative answer to the second question
about biconjugacy, which is stated as Theorem~\ref{THbiconjIC} in Section \ref{SCbiconj}.

Integral biconjugacy is closely related to 
integral subgradients.
For a point $x \in \domZ f$, 
the \kwd{integral subdifferential} of $f$ at $x$
is defined as 
\begin{equation} \label{subgZZdef}
 \subgZ f(x)
= \{ p \in  \ZZ\sp{n} \mid    
  f(y) - f(x)  \geq  \langle p, y - x \rangle     \ \mbox{for all }  y \in \ZZ\sp{n} \} ,
\end{equation}
and an element of $\subgZ f(x)$ is called an 
\kwd{integral subgradient}
of $f$ at $x$.
It is known that
$f\sp{\bullet\bullet}(x) = f(x)$ if and only if $\subgZ f(x) \not= \emptyset$;
see Lemma \ref{LMbiconjsubg} in Section \ref{SCbiconj}.
The condition $\subgZ f(x) \not= \emptyset$ is sometimes referred to as
the \kwd{integral subdifferentiability}  of $f$ at $x$.
Our proof of the integral biconjugacy
actually consists in showing the integral subdifferentiability,
which is stated as Theorem~\ref{THsubgrIC} in Section \ref{SCsubr}.

We can name the following significances of the present result:

\begin{enumerate}
\item
Our result of integral biconjugacy for integrally convex functions serves as a
unified proof of integral biconjugacy
for various classes of discrete convex functions,
such as 
{\rm L}-convex, ${\rm L}^{\natural}$-convex, {\rm M}-convex,  ${\rm M}^{\natural}$-convex,  
${\rm L}^{\natural}_{2}$-convex, and ${\rm M}^{\natural}_{2}$-convex functions.
The existing proofs for these functions 
are based on conjugacy statements valid for respective function classes,
and as such, vary with function classes.
Our proof considers integral biconjugacy directly, 
without involving conjugacy properties
that depend on function classes.

\item
In addition to being a unified proof for known results,
our result reveals new facts that
integer-valued BS-convex and UJ-convex functions
admit integral subgradients and  enjoy integral biconjugacy
(Corollaries \ref{COsubgrBSUJ} and \ref{CObiconjBSUJ}).

\item
Our results imply that a theory of discrete DC functions can be developed for 
integrally convex functions.
In particular, an analogue of the Toland--Singer duality
for integrally convex functions can be established.
See Section \ref{SCdcprog} for details.
\end{enumerate}

This paper is organized as follows.
Section~\ref{SCintcnvfn} is a review of relevant results on
integrally convex functions.
Section~\ref{SCres} presents the main results of this paper,
followed by Section~\ref{SCproof} for the proofs. 
Section~\ref{SCconclrem} concludes the paper with some remarks.

\section{Integrally Convex Functions}
\label{SCintcnvfn}

In this section we summarize fundamental facts about integrally convex functions.
The reader is referred to \cite{FT90} and \cite[Section 3.4]{Mdcasiam} for backgrounds.

For integer vectors 
$a \in (\ZZ \cup \{ -\infty \})\sp{n}$ and 
$b \in (\ZZ \cup \{ +\infty \})\sp{n}$ with $a \leq b$,
$[a,b]_{\ZZ}$ denotes the integer interval 
(box, discrete rectangle)
between $a$ and $b$, i.e.,
$[a,b]_{\ZZ} = \{ x \in \ZZ\sp{n} \mid a \leq x \leq b \}$.
For $x \in \RR^{n}$ the integral neighborhood of $x$ is defined as 
\[
N(x) = \{ z \in \mathbb{Z}^{n} \mid | x_{i} - z_{i} | < 1 \ (i=1,\ldots,n)  \}.
\]
For a function
$f: \mathbb{Z}^{n} \to \mathbb{R} \cup \{ +\infty  \}$
the local convex extension 
$\tilde{f}: \RR^{n} \to \RR \cup \{ +\infty \}$
of $f$ is defined 
as the union of all convex envelopes of $f$ on $N(x)$.  That is,
\begin{equation} \label{fnconvclosureloc2}
 \tilde f(x) = 
  \min\{ \sum_{y \in N(x)} \lambda_{y} f(y) \mid
      \sum_{y \in N(x)} \lambda_{y} y = x,  \ 
  (\lambda_{y})  \in \Lambda(x) \}
\quad (x \in \RR^{n}) ,
\end{equation} 
where $\Lambda(x)$ denotes the set of coefficients for convex combinations indexed by $N(x)$:
\[ 
  \Lambda(x) = \{ (\lambda_{y} \mid y \in N(x) ) \mid 
      \sum_{y \in N(x)} \lambda_{y} = 1, 
      \lambda_{y} \geq 0 \ (\forall y \in N(x))  \} .
\] 
If $\tilde f$ is convex on $\RR^{n}$,
then $f$ is said to be {\em integrally convex} \cite{FT90,Mdcasiam}.

A set $S \subseteq \ZZ^{n}$ is said to be 
integrally convex if
the convex hull $\overline{S}$ of $S$ coincides with the union of the 
convex hulls of $S \cap N(x)$ over $x \in \RR^{n}$,
i.e., if, for any $x \in \RR^{n}$, 
$x \in \overline{S} $ implies $x \in  \overline{S \cap N(x)}$.
A set $S$ is integrally convex if and only if its indicator function
$\delta_{S}: \ZZ\sp{n} \to \{ 0, +\infty \}$
is an integrally convex function,
where the indicator function $\delta_{S}$ is defined by
$\delta_{S}(x)  =
   \left\{  \begin{array}{ll}
    0            &   (x \in S) ,      \\
   + \infty      &   (x \not\in S) . \\
                      \end{array}  \right.$ 
An integrally convex set $S$ is ``hole-free'' in the sense that 
\begin{equation}  \label{holefree}
S =  \overline{S} \cap \mathbb{Z}^{n}.
\end{equation}

In this paper we need the following property of integrally convex sets.

\begin{proposition}    \label{PRpolyhedICset}
The convex hull $\overline{S}$
of an integrally convex set $S \subseteq \ZZ\sp{n}$ is
an integer polyhedron.
Moreover, for any face $F$ of $\overline{S}$,
the smallest affine subspace containing $F$ 
is given as 
$\{ x + \sum_{k=1}\sp{h} c_{k} d^{(k)} \mid c_{1}, c_{2}, \ldots, c_{h} \in \RR \}$
for a point $x$ in $F$ and some direction vectors 
$d^{(k)} \in \{ -1,0,+1 \}\sp{n}$ $(k=1,2,\ldots, h)$.
\end{proposition}
\begin{proof}
The proof is given in Section \ref{SCproofpolyh}.
\end{proof}

\begin{remark} \rm \label{RMedgedir}
The properties mentioned in Proposition \ref{PRpolyhedICset} do not characterize 
integral convexity of a set.
For example, let
$S = \{  (0,0,0), \allowbreak  (1,0,1),  \allowbreak (1,1,-1), (2,1,0) \}$.
The convex hull $\overline{S}$ is a parallelogram with edge directions
$(1,0,1)$ and $(1,1,-1)$,
and hence is an integer polyhedron such that
the smallest affine subspace containing each face is spanned by $\{ -1,0,1 \}$-vectors.
However, $S$ is not integrally convex, since
$x = [(1,0,1) + (1,1,-1) ]/2 = (1,1/2,0) \in \overline{S}$,  
$N(x) = \{ (1,0,0), (1,1,0) \}$, and 
$S \cap N(x) = \emptyset$.
\finbox
\end{remark}

The effective domain of an integrally convex function is an integrally convex set.
Integral convexity of a function can be characterized by a local condition
under the assumption that the effective domain is an integrally convex set.

\begin{theorem}[\cite{FT90,MMTT17proxIC}]
\label{THfavtarProp33}
Let $f: \mathbb{Z}^{n} \to \mathbb{R} \cup \{ +\infty  \}$
be a function with an integrally convex effective domain.
Then the following properties are equivalent:

{\rm (a)}
$f$ is integrally convex.

{\rm (b)}
For every $x, y \in \ZZ\sp{n}$ with $\| x - y \|_{\infty} =2$  we have \ 
\begin{equation}  \label{intcnvconddist2}
\tilde{f}\, \bigg(\frac{x + y}{2} \bigg) 
\leq \frac{1}{2} (f(x) + f(y)).
\end{equation}
\vspace{-1.7\baselineskip} \\
\finbox
\end{theorem}

A minimizer of an integrally convex function
can be characterized by a local minimality condition as follows.

\begin{theorem}[\protect{\cite[Proposition 3.1]{FT90}};
  see also \protect{\cite[Theorem 3.21]{Mdcasiam}}]
  \label{THintcnvlocopt}
Let $f: \mathbb{Z}^{n} \to \mathbb{R} \cup \{ +\infty  \}$
be an integrally convex function and $x^{*} \in \domZ f$.
Then $x^{*}$ is a minimizer of $f$ 
if and only if
$f(x^{*}) \leq f(x^{*} +  d)$ for all $d \in  \{ -1, 0, +1 \}^{n}$.
\finbox
\end{theorem}

\begin{remark} \rm \label{RMintcnvconcept}
The concept of integrally convex functions is introduced in \cite{FT90} 
for functions defined on integer intervals (discrete rectangles).
The extension to functions with general integrally convex effective domains
is straightforward, which is found in \cite{Mdcasiam}.
Theorem~\ref{THfavtarProp33} is proved in \cite[Proposition 3.3]{FT90} 
when the effective domain is an integer interval 
and in \cite{MMTT17proxIC} for the general case. 
\finbox
\end{remark}

\begin{remark} \rm \label{RMconjIC}
The integral conjugate of an integrally convex function $f$ is not necessarily integrally convex.
This is shown by the following example
(\cite[Example 4.15]{MS01rel} with a minor correction).
Let $S = \{(1, 1, 0, 0), (0, 1, 1, 0), (1, 0, 1, 0), (0, 0, 0, 1)\}$.
This is obviously an integrally convex set, as it is contained in $\{ 0,1 \}\sp{4}$.
Accordingly, its indicator function $\delta_{S}: \ZZ^{4} \to \{0, + \infty\}$ 
is integrally convex.
 The integral conjugate $g = \delta_{S}^\bullet$ is given by
\[
 g(p_{1}, p_{2}, p_{3}, p_{4}) 
 = \max\{p_{1} + p_{2}, p_{2} + p_{3}, p_{1} + p_{3}, p_{4}\} \qquad (p \in \ZZ^{4}).
\]
Let $\tilde g$ be the local convex extension of $g$.
For $p =(0,0,0,0)$ and $q=(1,1,1,2)$ 
we have  
$(p+q)/2 =(1/2,1/2,1/2,1) = [(1,0,0,1) + (0,1,0,1) + (0,0,1,1) + (1,1,1,1)]/4$ and 
$\tilde g ((p+q)/2)  = [g(1,0,0,1) + g(0,1,0,1) + g(0,0,1,1) + g(1,1,1,1)]/4 
= (1+1+1+2)/4 = 5/4$,
whereas $(g(p)+ g(q)) /2 = (0+2)/2 = 1$.
Thus we have
$\tilde g ((p+q)/2) > (g(p)+ g(q)) /2$,
violating (\ref{intcnvconddist2}) in Theorem~\ref{THfavtarProp33}.
Hence $g$ is not integrally convex.
\finbox
\end{remark}

\section{Results}
\label{SCres}

\subsection{Integral subgradients}
\label{SCsubr}

\begin{theorem}[Integral subdifferentiability]  \label{THsubgrIC}
For an integer-valued integrally convex function 
$f: \ZZ^{n} \to \ZZ \cup \{ +\infty \}$,
we have $\subgZ f(x) \neq \emptyset$ for all $x \in\domZ f$.
\end{theorem}
\begin{proof}
The proof is given in Section \ref{SCproofsubgr}.
\end{proof}

The following example shows that integral subdifferentiability is not guaranteed
without the assumption of integral convexity.

\begin{example}[{\cite[Example 1.1]{Mdca98}}]  \rm \label{EXla1}
Let $D = \{ (0,0,0), \pm (1,1,0), \pm (0,1,1), \pm (1,0,1) \}$
and $f: \mathbb{Z}^3 \to \mathbb{Z} \cup \{+\infty\}$ be defined by
\begin{align*}
f(x_{1},x_{2},x_{3}) =
 \begin{cases} (x_{1}+x_{2}+x_{3})/2 & (x \in D), \\
                +\infty & (\textrm{otherwise}).  
  \end{cases}
\end{align*}
This function can be naturally extended to a convex function on 
the convex hull $\overline{D}$ of $D$
and $D$ is hole-free in the sense of (\ref{holefree}).
However, $D$ is not integrally convex since
for $x=(1,1,0)$ and $y=(-1,0,-1)$ we have 
$(x+y)/2 = (0, 1/2, -1/2)$,
$N((0, 1/2, -1/2)) = \{ (0,0,0), (0,1,0), (0,0,-1), (0,1,-1)  \}$,
and hence
$N((0, 1/2, -1/2)) \cap D    \allowbreak   = \{ (0,0,0) \}$.
Therefore, $f$ is not integrally convex.

To investigate the integral subgradient of $f$ at the origin,
suppose that $p \in \subgZ f(\veczero) \subseteq \mathbb{Z}^3$. 
Since $f(y) - f(\veczero) \ge \langle p, y \rangle$ for all $y \in D$, we must have
\begin{center}
\begin{tabular}{ccc}
$ 1 \ge  p_{1} + p_{2}$, & $ 1 \ge  p_{2} + p_{3}$, & $ 1 \ge  p_{3} + p_{1}$, \\
$-1 \ge -p_{1} - p_{2}$, & $-1 \ge -p_{2} - p_{3}$, & $-1 \ge -p_{3} - p_{1}$.
\end{tabular}
\end{center}
However, this system admits no integer solution, 
though it is satisfied by $(p_{1}, p_{2}, p_{3}) = (1/2, 1/2, 1/2)$.
Hence $\subgZ f(\veczero) = \emptyset$.
\finbox
\end{example}

\begin{remark} \rm \label{RMsubgNotIntPolyh}
Here is a subtle point about the statement of Theorem~\ref{THsubgrIC}. 
In parallel to the integral subdifferential $\subgZ f(x)$ in (\ref{subgZZdef}),
the (real) subdifferential $\subgR f(x)$ is defined by
\begin{equation} \label{subgZRdef2}
 \subgR f(x)
= \{ p \in  \RR\sp{n} \mid    
  f(y) - f(x)  \geq  \langle p, y - x \rangle     \ \mbox{for all }  y \in \ZZ\sp{n} \}.
\end{equation}
Theorem~\ref{THsubgrIC} states that 
$\subgR f(x) \cap \ZZ\sp{n} \not= \emptyset$,
but it does not claim a stronger statement that $\subgR f(x)$ is an integer polyhedron.
Indeed, $\subgR f(x)$ is not necessarily an integer polyhedron, as the following example shows.
Let 
$f: \mathbb{Z}^3 \to \mathbb{Z} \cup \{+\infty\}$ be defined by
$f(0,0,0)=0$ and $f(1,1,0)=f(0,1,1)=f(1,0,1)=1$,
with $\domZ f = \{ (0,0,0), (1,1,0), (0,1,1), (1,0,1) \}$.
This function is integrally convex
 and the subdifferential of $f$ at the origin is given as
\[
\subgR f(\veczero) = \{ p= (p_{1}, p_{2}, p_{3}) \in \RR\sp{3} \mid 
 p_{1} + p_{2} \leq 1,  
 p_{2} + p_{3} \leq 1,  
 p_{1} + p_{3} \leq 1   \},
\]
which is not an integer polyhedron, having a non-integral vertex at $p=(1/2, 1/2, 1/2)$.
In contrast, it is known \cite{Mdcasiam} that,
$\subgR f(x)$ is an integer polyhedron if $f$ is 
${\rm L}^{\natural}$-convex, ${\rm M}^{\natural}$-convex,  
${\rm L}^{\natural}_{2}$-convex, or ${\rm M}^{\natural}_{2}$-convex.
\finbox
\end{remark}

\begin{remark} \rm \label{RMsubgBS}
BS-convex and UJ-convex functions are investigated by Fujishige \cite{Fuj14bisubmdc}.
For an integer-valued BS-convex function $f$, 
the subdifferential $\subgR f(x)$ in (\ref{subgZRdef2})
contains a half-integral vector
\cite[Theorem 2]{Fuj14bisubmdc}, and it contains an integral vector
if the function $f$ arises as the conjugate of a $D$-convex function,
which, by definition, is associated with a disconcordant Freudenthal simplicial division $D$ 
\cite[Theorem 5]{Fuj14bisubmdc}.
The function used as an example in Remark \ref{RMsubgNotIntPolyh} 
is actually a BS-convex function (\cite[Example~3]{Fuj14bisubmdc}),
and therefore, $\subgR f(x)$ is not necessarily an integer polyhedron
for a BS-convex function $f$.
Nevertheless, BS-convex and UJ-convex functions admit integral subgradients,
as they are integrally convex.  This fact is stated below as a corollary
of Theorem~\ref{THsubgrIC}.
\finbox
\end{remark}

\begin{corollary}  \label{COsubgrBSUJ}
\quad

\noindent
{\rm (1)}
For an integer-valued BS-convex function $f$, 
we have $\subgZ f(x) \neq \emptyset$ for all $x \in\domZ f$.

\noindent
{\rm (2)}
For an integer-valued UJ-convex function $f$, 
we have $\subgZ f(x) \neq \emptyset$ for all $x \in\domZ f$.
\finbox
\end{corollary}

\subsection{Integral biconjugacy}
\label{SCbiconj}

In this section we establish the integral biconjugacy
$f\sp{\bullet\bullet} = f$
for integer-valued integrally convex functions.

\begin{lemma}[{\cite[Lemma 4.1]{Mdca98}}]  \label{LMbiconjsubg}
 For each $x \in \domZ f$ we have: \ \ 
$f\sp{\bullet\bullet}(x) = f(x)  \iff  \subgZ f(x) \not= \emptyset$. 
\end{lemma}
\begin{proof}
By the definitions (\ref{conjvexZpZ}) and (\ref{subgZZdef})
it holds, for $x \in \domZ f$ and $p \in \ZZ\sp{n}$, that
\begin{equation} \label{subgconj}
p \in  \subgZ f(x)  \iff f(x) + f\sp{\bullet}(p) = \langle p, x \rangle  .
\end{equation}
If there exists $p \in \subgZ f(x)$, (\ref{subgconj}) implies that 
$f(x) + f\sp{\bullet}(p) = \langle p, x \rangle$.   From this 
and the definition of $f\sp{\bullet\bullet}(x)$ we obtain
$f\sp{\bullet\bullet}(x) \geq \langle p, x \rangle - f\sp{\bullet}(p) = f(x)$,
while $ f\sp{\bullet\bullet}(x) \leq f(x)$ is obvious.
Conversely, if 
$f\sp{\bullet\bullet}(x) = f(x)$, 
there exists 
$p \in \ZZ\sp{n}$ such that
$\langle p, x \rangle - f\sp{\bullet}(p) = f\sp{\bullet\bullet}(x) = f(x)$.
This implies  $p \in \subgZ f(x)$ by (\ref{subgconj}).
\end{proof}

\begin{remark}  \rm  \label{RMfbbf}
The desired integral biconjugacy $f\sp{\bullet\bullet} = f$
does not follow immediately from the combination of 
Theorem~\ref{THsubgrIC} and Lemma \ref{LMbiconjsubg}.
Let 
$f: \ZZ\sp{2} \to \ZZ \cup \{+\infty \}$ be the indicator function of 
$S = \{ (x_{1},x_{2}) \in \ZZ\sp{2}   \mid x_{2} \geq \sqrt{2} x_{1} - 1/2 \}$, which is not an integrally convex set. 
Then  $\subgZ f(x) = \subgR f(x) = \{ (0,0) \} \not= \emptyset$
for all $x \in S = \domZ f$.
On the other hand,
we have $f\sp{\bullet}(0,0) = 0$ and $f\sp{\bullet}(p) =+\infty$ 
for $p \in \ZZ\sp{2}\setminus \{(0,0)\}$,
from which follows that
$f\sp{\bullet\bullet}(x)= 0$ for all $x \in \ZZ\sp{2}$.
Thus we have $\domZ f\sp{\bullet\bullet} =  \ZZ\sp{2} \not= \domZ f$,
and, a fortiori,  $f\sp{\bullet\bullet} \not= f$.
This example,
taken from \cite[Remark 4.1]{Mdca98}, motivates the technical condition (\ref{Fcond2}) below.
\finbox
\end{remark}

For $f: \mathbb{Z}^{n} \to \mathbb{Z} \cup \{+\infty\}$
we consider the following conditions:
\begin{align}
 &  \domZ f = {\rm cl}(\overline{\domZ f}) \cap \mathbb{Z}^{n}  \neq \emptyset,
\label{Fcond1} \\
    & {\rm cl}(\overline{\domZ f}) \textrm{ is rationally-polyhedral},
\label{Fcond2} \\
    & \subgZ f(x) \neq \emptyset
 \ \mbox{for all} \  x \in \domZ f ,
\label{Fcond3}
\end{align}
where ${\rm cl}(\overline{\domZ f})$ denotes 
the closure of the convex hull%
\footnote{
${\rm cl}(\overline{\domZ f})$ coincides with the closed convex hull of $\domZ f$;
\cite[Section 1.4]{HL01}.
} 
of $\domZ f$,
and a closed convex set in $\mathbb{R}^{n}$ is said to be
rationally-polyhedral if it is described by a system of finitely many
inequalities with coefficients of rational numbers.
The first condition (\ref{Fcond1}) is natural,
the second condition (\ref{Fcond2}) is rather technical, 
and the third condition (\ref{Fcond3}) is essential.

\begin{lemma}[{\cite[Lemma 4.2]{Mdca98}}]    \label{LMsubgext}
Suppose that 
$f: \mathbb{Z}^{n} \to \mathbb{Z} \cup \{+\infty\}$
satisfies the conditions {\rm (\ref{Fcond1})}, {\rm (\ref{Fcond2})}, and {\rm (\ref{Fcond3})}.%
\footnote{
In Lemma 4.1 of \cite{Mdca98} an additional condition 
``$\partial_\mathbb{R} f(x) = {\rm cl}(\overline{\domZ f})$''
is involved in the definition of $\mathcal{F}_G$ in (4.18).
However, we can verify that this condition is not needed.
} \  
Then the  following hold.
\smallskip
\par
\noindent
{\rm (1)} \
${\displaystyle  
 \domZ f\sp{\bullet} =  \bigcup \{ \subgZ f(x) \mid x \in \domZ f \}
 \not= \emptyset.
}$

\smallskip

\noindent
{\rm (2)} \
$ \domZ f\sp{\bullet\bullet} = \domZ f$.

\smallskip

\noindent
{\rm (3)} \
$f\sp{\bullet\bullet}(x) = f(x) \qquad (x \in \ZZ\sp{n})$.    

\smallskip

\noindent
{\rm (4)} \
For $x \in \domZ f$, $p \in \domZ f\sp{\bullet}:$  \quad
$p \in  \subgZ f(x) \iff x \in  \subgZ f\sp{\bullet}(p)$.

\smallskip

\noindent
{\rm (5)} \
$\subgZ f\sp{\bullet}(p) \not= \emptyset 
\qquad (p \in \domZ f\sp{\bullet})$.
\finbox
\end{lemma}

\begin{lemma}    \label{LMcond123IC}
An integer-valued integrally convex function satisfies 
the conditions {\rm (\ref{Fcond1})}, {\rm (\ref{Fcond2})}, and {\rm (\ref{Fcond3})}.
\end{lemma}
\begin{proof}
Since  $S = \domZ f$ is integrally convex,
$\overline{S}$ is an integer polyhedron by Proposition \ref{PRpolyhedICset}.
In particular, we have
${\rm cl}(\overline{S}) = \overline{S}$.
The condition (\ref{Fcond1}) is satisfied by (\ref{holefree}).
The property (\ref{Fcond2}) can be shown as follows.
By Proposition \ref{PRpolyhedICset},
the smallest affine subspace containing a facet $F$ of $\overline{S}$
is described by a system of equations, say,  $A_{F} x = b_{F}$
with the entries of $A_{F}$ belonging to $\{ -1,0,+1 \}$
and $b_{F}$ being an integer vector.
This implies the rationality  (\ref{Fcond2}).
The property (\ref{Fcond3}) is shown in Theorem~\ref{THsubgrIC}.
\end{proof}

By combining Lemmas \ref{LMsubgext} and \ref{LMcond123IC}, 
we obtain the following statements about 
the integral subdifferential, integral conjugate, and integral biconjugate
of an integer-valued integrally convex function.

\begin{proposition}    \label{PRsubgextIC}
For an integer-valued integrally convex function 
$f: \mathbb{Z}^{n} \to \mathbb{Z} \cup \{+\infty\}$,
we have the properties {\rm (1)} to {\rm (5)} in Lemma {\rm \ref{LMsubgext}}.
\finbox
\end{proposition}

The integral biconjugacy claimed in Proposition \ref{PRsubgextIC}
deserves a separate statement as a theorem.

\begin{theorem}[Integral biconjugacy]  \label{THbiconjIC}
For an integer-valued integrally convex function 
$f: \ZZ^{n} \to \ZZ \cup \{ +\infty \}$
we have
$f\sp{\bullet\bullet}(x) =f(x)$ for all $x \in \ZZ\sp{n}$.
\finbox
\end{theorem}

The following example shows that
integral biconjugacy is not guaranteed
without the assumption of integral convexity.

\begin{example}[{\cite[Example 1.1]{Mdca98}}]  \rm \label{EXla1cont}
In Example \ref{EXla1}, $D = \domZ f$ is not an integrally convex set,
and therefore $f$ is not integrally convex. 
The integral conjugate of $f$ is given as 
\[
f^{\bullet}(p) = \max \{ 0, |p_{1}+p_{2}-1|, |p_{2}+p_{3}-1|, |p_{3}+p_{1}-1| \}
\]
and the integral biconjugate is
$f^{\bullet\bullet}(x) = \sup_{p \in \mathbb{Z}^3} \{ \langle p, x \rangle - f^{\bullet}(p) \}$.
Hence 
\[
f^{\bullet\bullet}(\veczero) = - \inf_{p \in \mathbb{Z}^3} \max \{ 0, |p_{1}+p_{2}-1|, |p_{2}+p_{3}-1|, |p_{3}+p_{1}-1| \}.
\]
Therefore we have $f^{\bullet\bullet}(\veczero) = -1 \neq 0 = f(\veczero)$. 
This shows $f^{\bullet\bullet} \neq f$.
\finbox
\end{example}

As special cases of Theorem~\ref{THbiconjIC}
we obtain integral biconjugacy 
for {\rm L}-convex,
${\rm L}^{\natural}$-convex,
{\rm M}-convex,  
${\rm M}^{\natural}$-convex,  
${\rm L}^{\natural}_{2}$-convex, and 
${\rm M}^{\natural}_{2}$-convex functions
given in \cite[Theorems 8.12, 8.36, 8.46]{Mdcasiam}.
Integral biconjugacy for BS-convex and UJ-convex functions 
are also obtained as a corollary of Theorem~\ref{THbiconjIC}.

\begin{corollary}  \label{CObiconjBSUJ}
\quad

\noindent
{\rm (1)}
For an integer-valued BS-convex function $f$, 
we have
$f\sp{\bullet\bullet}(x) =f(x)$ for all $x \in \ZZ\sp{n}$.

\noindent
{\rm (2)}
For an integer-valued UJ-convex function $f$, 
we have
$f\sp{\bullet\bullet}(x) =f(x)$ for all $x \in \ZZ\sp{n}$.
\end{corollary}

\subsection{Discrete DC programming}
\label{SCdcprog}

A discrete analogue of 
the theory of DC functions (difference of two convex functions) 
and DC programming has recently been proposed in \cite{MM15dcprog}
using \Lnat-convex and \Mnat-convex functions.
As already noted in \cite[Remark 4.7]{MM15dcprog},
such  theory of discrete DC functions can in fact be developed for functions
that satisfy integral biconjugacy and integral subdifferentiability.
Our present results, Theorems \ref{THsubgrIC} and \ref{THbiconjIC},
enable us to extend the theory of discrete DC functions 
to integrally convex functions.
In particular,
an analogue of 
the Toland--Singer duality~\cite{Sin79,Tol79}
can be established for integrally convex functions as a corollary of our results.

\begin{theorem}[Toland--Singer duality] \label{THtolandsinger}
Let $g$ and $h$ be integer-valued integrally convex functions.%
\footnote{
As the proof shows, the integral convexity of $g$ is not needed.
That is, (\ref{tolandsingerduality}) holds for 
any  $g: \mathbb{Z}^{n} \to \mathbb{Z} \cup \{+\infty\}$,
as long as $h: \mathbb{Z}^{n} \to \mathbb{Z} \cup \{+\infty\}$ is integrally convex.
} 
Then
\begin{align} \label{tolandsingerduality}
\inf_{x \in \mathbb{Z}^{n}} \{ g(x) - h(x) \} 
= \inf_{p \in \mathbb{Z}^{n}} \{ h^{\bullet}(p) - g^{\bullet}(p) \} . 
\end{align}
\end{theorem}
\begin{proof}
By integral biconjugacy (Theorem~\ref{THbiconjIC}) of $h$,
we can prove (\ref{tolandsingerduality}) as follows:
\begin{align*}
& \inf_x \{ g(x) - h(x) \} = \inf_x \{ g(x) - h^{\bullet\bullet}(x) \} 
        = \inf_x \{ g(x) - \sup_p \{ \left<p,x\right> - h^{\bullet}(p) \} \} 
\\
  &= \inf_x \inf_p \{ g(x) - \left<p, x\right> + h^{\bullet}(p) \} 
= \inf_p \{ h^{\bullet}(p) - \sup_x \{ \left<p, x\right> - g(x) \} \} 
\\
& = \inf_p \{ h^{\bullet}(p) - g^{\bullet}(p) \} .
\end{align*}
\vspace{-2\baselineskip} \\
\end{proof}


\section{Proofs}
\label{SCproof}

\subsection{Proof of Proposition \ref{PRpolyhedICset} about the convex hull}
\label{SCproofpolyh}

We start with a basic fact, which will be intuitively obvious.
\begin{lemma}  \label{LMhullICclosed}
The convex hull $\overline{S}$ of an integrally convex set $S$ is a closed set.
\end{lemma}
\begin{proof}
Take any point $x$ in the (topological) closure of $\overline{S}$.
There exists a sequence 
$\{ x_{k} \} \subseteq \overline{S}$ that converges to $x$.
We may assume that $N(x) \subseteq N(x_{k})$ holds for all $k$,
by considering a subsequence consisting of $\{ x_{k} \}$ with
$\| x_{k} - x \|_{\infty} < \varepsilon$ for a sufficiently small $\varepsilon >0$.
We may further assume that $N(x_{k})$  is identical for all $k$,
since there are finitely many possibilities of the set $N(x_{k})$  
and we can choose an appropriate subsequence of $\{ x_{k} \}$.
Let $N_{*}$ denote this $N(x_{k})$.
Since $S$ is integrally convex and 
$x_{k} \in \overline{S}$,
we have
$x_{k} \in 
\overline{S \cap N(x_{k})} = \overline{S \cap N_{*}}$.
Here $\overline{S \cap N_{*}}$ is a closed set, since $S \cap N_{*}$ is a finite set. 
Therefore, $x = \lim_{k} x_{k} \in \overline{S \cap N_{*}} \subseteq \overline{S}$.
\end{proof}

Let $S \subseteq \ZZ\sp{n}$ be an integrally convex set,
and $F$ be a face of its convex hull $\overline{S}$.
Let $L_{F}$ denote the linear subspace of $\RR\sp{n}$ 
such that the smallest affine subspace containing $F$
is represented as $ x + L_{F}$ for a point $x$ in $F$.
In the following we prove Proposition \ref{PRpolyhedICset} by showing that
(1) $F$ contains an integer point,
(2) $L_{F}$ is spanned by vectors in $\{-1,0,+1\}\sp{n}$, and
(3) $\overline{S}$ is a polyhedron.

\medskip

Proof of (1):
Take any $x \in \RR\sp{n}$ in $F$.
By the integral convexity of $S$, we have
$x \in \overline{S \cap N(x)}$.  That is,
there exist integer points $y\sp{(1)}, y\sp{(2)}, \ldots, y\sp{(m)} \in S \cap N(x)$
and  $\lambda_{1}, \lambda_{2}, \ldots, \lambda_{m} > 0$ such that
$ x = \sum_{k=1}\sp{m} \lambda_{k}y\sp{(k)}$
and $\sum_{k=1}\sp{m} \lambda_{k} = 1$.
Here we have
$y\sp{(1)}, y\sp{(2)},  \allowbreak \ldots,   \allowbreak y\sp{(m)} \in F$,
since $F$ is a face of $\overline{S}$, $x \in F$, and
$y\sp{(1)}, y\sp{(2)},  \ldots,  y\sp{(m)} \in \overline{S}$.

Proof of (2):
Fix $x \in F \cap \ZZ\sp{n}$.
We shall show that there exist $d^{(1)}, d^{(2)}, \ldots, d^{(h)} \in \{-1,0,+1\}\sp{n}$
such that
\begin{equation}  \label{prfFAC1}
 F = (x + \mbox{span}\{ d^{(1)}, d^{(2)},\ldots, d^{(h)} \}) \cap \overline{S},
\end{equation}
where
$\mbox{span}\{ \cdots \}$
means the subspace spanned by the vectors in the braces.
We assume that $F$ is not a singleton,
since otherwise (\ref{prfFAC1}) is trivially true.
Take any $y \in F \setminus \{ x \}$ and
define
$z = (1-\varepsilon)x + \varepsilon y$
with a sufficiently small $\varepsilon > 0$
so that $x \in N(z)$.
Since $z \in \overline{S}$
and $S$ is integrally convex,
there exist $z\sp{(1)}, z\sp{(2)}, \ldots, z\sp{(m)} \in S \cap N(z)$
and  $\lambda_{1}, \lambda_{2}, \ldots, \lambda_{m} > 0$ such that
$ z = \sum_{k=1}\sp{m} \lambda_{k}z\sp{(k)}$ and
$\sum_{k=1}\sp{m} \lambda_{k} = 1$.
Here we have
$z\sp{(1)}, z\sp{(2)}, \ldots, z\sp{(m)} \in F$, since $F$ is a face of $\overline{S}$,
$z \in F$,
and $z\sp{(1)}, z\sp{(2)}, \ldots, z\sp{(m)} \in \overline{S}$.
It follows from 
 $(1-\varepsilon)x + \varepsilon y = z = \sum_{k=1}\sp{m} \lambda_{k}z\sp{(k)}$
that
\[
 y = x + \frac{1}{\varepsilon} \sum_{k=1}\sp{m} \lambda_{k}(z\sp{(k)} - x),
\]
where each direction vector
$z\sp{(k)} - x$ belongs to $\{-1,0,+1\}\sp{n}$,
since both $z\sp{(k)}$ and $x$ are members of $N(z)$.
By collecting all the direction vectors
$z\sp{(k)} - x$ arising from all choices of $y \in F \setminus \{ x \}$
we obtain a set of vectors $\{ d^{(1)}, d^{(2)}, \ldots, d^{(h)} \}  \subseteq \{-1,0,+1\}\sp{n}$ 
for which (\ref{prfFAC1}) holds.

Proof of (3):
First suppose that $\overline{S}$ is full dimensional.
For a facet $F$ of $\overline{S}$,
the linear subspace $L_{F}$ is a hyperplane of dimension $n-1$,
and is described by an (outward) normal vector.
The normal vector is perpendicular to 
$(n-1)$ linearly independent direction vectors generating $L_{F}$
and is uniquely determined under some appropriate normalization of the length.
Since the direction vectors are contained in $\{-1,0,+1\}\sp{n}$ 
by (\ref{prfFAC1}),  there exist only a finite number of possible normal vectors, and hence 
$\overline{S}$ has a finite number of facets.
If $\overline{S}$ is not full dimensional,
we consider normal vectors of its facets contained in 
the subspace $L_{\overline{S}}$.
There are only a finite number of such normal vectors, up to scaling.
Therefore, $\overline{S}$ is a polyhedron.


\subsection{Proof of Theorem~\ref{THsubgrIC} for integral subdifferentiability}
\label{SCproofsubgr}

Let 
$f: \mathbb{Z}^{n} \to \mathbb{Z} \cup \{ +\infty  \}$
be an integer-valued integrally convex function.
For a point $x \in \domZ f$, 
the subdifferential of $f$ at $x$
is defined as 
\begin{equation} \label{subgZRdef}
 \subgR f(x)
= \{ p \in  \RR\sp{n} \mid    
  f(y) - f(x)  \geq  \langle p, y - x \rangle     \ \mbox{for all }  y \in \ZZ\sp{n} \}.
\end{equation}
The subdifferential $\subgR f(x)$ is nonempty for every $x \in \domZ f$, 
since an integrally convex function is extensible to a convex function.
In the following we prove that $\subgR f(x)$ contains an integer vector,
which is the claim of Theorem~\ref{THsubgrIC}.

We may assume that $x = \veczero$ and $f(\veczero) = 0$.
In the definition of  $\subgR f(\veczero)$ by (\ref{subgZRdef}),
it suffices, by Theorem~\ref{THintcnvlocopt},
to consider  $y$ in $\{-1,0,+1\}^n$.
Therefore, we have
\begin{equation} \label{subgZR0def}
 \subgR f(\veczero)
= \{ p \in  \RR\sp{n} \mid    
  \sum_{j=1}\sp{n} y_{j} p_{j} \leq f(y)   \ \mbox{for all }  y \in \{ -1,0,+1 \}\sp{n} \} .
\end{equation}
We represent the system of inequalities 
$\sum_{j=1}\sp{n} y_{j} p_{j} \leq f(y)$ 
for $y$ with $f(y) < +\infty$
in a matrix form as
\begin{equation}\label{ineqApb}
 A p \leq b.
\end{equation}
Let $I$ denote the row set of $A$ and $A = ( a_{ij} \mid i \in I, j \in \{ 1,2,\ldots, n \})$.
We denote the $i$th row vector of $A$ by $a_{i}$ for $i \in I$.
The row set $I$ is 
indexed by $y \in \{ -1,0,+1 \}\sp{n}$ with $f(y) < +\infty$,
and 
$a_{i}$ is equal to the corresponding $y$ for $i \in I$;
we have $a_{ij} = y_{j}$ for $j=1,2,\ldots, n$ and $b_{i}= f(a_{i})$.
Note that $a_{ij} \in \{ -1,0,+1 \}$
and $a_{i} \in \{ -1,0,+1 \}\sp{n}$ for all $i$ and $j$.

We apply the Fourier--Motzkin elimination procedure \cite{Sch86} to 
the system of inequalities (\ref{ineqApb})
to show the existence of an integer vector satisfying
(\ref{ineqApb}).

The Fourier--Motzkin elimination for (\ref{ineqApb}) goes as follows.
According to the value of the coefficient $a_{i1}$ of the first variable $p_{1}$,
we partition $I$ into three disjoint parts $(I_{1}^{+},I_{1}^{0},I_{1}^{-})$ as
\begin{align*}
 I_{1}^{+} &= \{ i \in I \mid a_{i1} = +1 \},   
\\ 
 I_{1}^{0} &= \{ i \in I \mid a_{i1} = 0 \},   
\\
 I_{1}^{-} &= \{ i \in I \mid a_{i1} = -1 \},
\end{align*}
and decompose (\ref{ineqApb}) into three parts as
\begin{align}
 a_{i}  p \leq b_{i} &
\qquad (i \in I_{1}^{+}),
\label{ineqApbi+}
\\
 a_{i}  p \leq b_{i} &
\qquad (i \in I_{1}^{0}) ,
\label{ineqApbi0}
\\
 a_{i}  p \leq b_{i} &
\qquad (i \in I_{1}^{-}) .
\label{ineqApbi-}
\end{align}
For all possible combinations of $i \in I_{1}^{+}$ and $k \in I_{1}^{-}$,
we add the inequality for $i$ in (\ref{ineqApbi+}) and the inequality for $k$ in (\ref{ineqApbi-}) 
to obtain
\begin{equation}\label{FMp2pn}
 (a_{i} + a_{k}) p \leq b_{i} + b_{k} 
\qquad (i \in I_{1}^{+},\; k \in I_{1}^{-}) .
\end{equation}
The inequalities in  (\ref{FMp2pn}) are free from the variable $p_{1}$,
since $a_{i1} + a_{k1}= 0$ for all
$i \in I_{1}^{+}$ and $k \in I_{1}^{-}$.
For the variable $p_{1}$ we obtain
\begin{equation}\label{FMp1}
 \max_{k \in I_{1}^{-}} \left\{ \sum_{j=2}^n a_{kj}p_{j} - b_{k} \right\} 
 \leq p_{1} \leq  
 \min_{i \in I_{1}^{+}} \left\{ b_{i} - \sum_{j=2}^n a_{ij}p_{j}  \right\} 
\end{equation}
from (\ref{ineqApbi+}) and (\ref{ineqApbi-}).
It is understood that
the maximum over the empty set is equal to $-\infty$ 
and the minimum over the empty set is equal to $+\infty$.

We have thus eliminated $p_{1}$ and obtained
a system of inequalities 
in $(p_{2},\ldots,p_{n})$ consisting of (\ref{ineqApbi0}) and (\ref{FMp2pn}).
Once $(p_{2},\ldots,p_{n})$ is found,
$p_{1}$ can easily be found from (\ref{FMp1}),
if the interval described by (\ref{FMp1}) is nonempty.
It is important that the derived system of inequalities 
in $(p_{1}, p_{2},\ldots,p_{n})$ 
consisting of  (\ref{ineqApbi0}), (\ref{FMp2pn}), and (\ref{FMp1})
is in fact equivalent to the original system consisting of
(\ref{ineqApbi+}), (\ref{ineqApbi0}), and (\ref{ineqApbi-}).
In particular, $(p_{1}, p_{2},\ldots,p_{n})$ satisfies 
(\ref{ineqApbi+}), (\ref{ineqApbi0}), and (\ref{ineqApbi-})
if and only if $(p_{2},\ldots,p_{n})$ satisfies 
(\ref{ineqApbi0}) and  (\ref{FMp2pn}),
and $p_{1}$ satisfies (\ref{FMp1}).

The Fourier--Motzkin elimination applies the above procedure 
recursively to eliminate variables $p_{1},p_{2}, \ldots,p_{n-1}$.
This process
results in a single inequality in $p_{n}$ of the form (\ref{FMp1}).
Then we can determine $(p_{1}, p_{2},\ldots,p_{n})$ 
in the reverse order $p_{n}, p_{n-1},\ldots,p_{1}$.

By virtue of the integral convexity of $f$, a drastic simplification occurs
in the elimination process. 
The inequalities (\ref{FMp2pn}) that are to be added in general 
are actually redundant and need not be added,
which is shown in the following lemma.
The lemma implies in particular that
$I_{1}^{0}$ is nonempty if 
$I_{1}^{+}$ and $I_{1}^{-}$ are nonempty.

\begin{lemma}  \label{LMelimIC}
The inequalities in {\rm (\ref{FMp2pn})} are implied by those in  {\rm (\ref{ineqApbi0})}.
\end{lemma}
\begin{proof}
In (\ref{FMp2pn})
we have
$b_{i} = f(a_{i})$ and $b_{k} = f(a_{k})$,
and hence the inequality in (\ref{FMp2pn}) can be rewritten as
\begin{equation}\label{FMp2pnvar}
 \frac{1}{2} (   a_{i}+ a_{k} ) p 
 \leq 
 \frac{1}{2} (   f(a_{i}) +f(a_{k}) ) .
\end{equation}

By the integral convexity of $f$
there exist 
$y\sp{(1)}, y\sp{(2)}, \ldots, y\sp{(m)} \in N((a_{i}+a_{k})/2)$ 
such that
\begin{equation}\label{prfFMredun1}
 \sum_{l=1}\sp{m} \lambda_{l} y\sp{(l)}
  = 
 \frac{1}{2} (   a_{i} + a_{k} ), 
\qquad
 \sum_{l=1}\sp{m} \lambda_{l} f(y\sp{(l)})
\leq
 \frac{1}{2} (   f(a_{i}) +f(a_{k}) ) ,
\end{equation}
where $\lambda_l \geq 0$ for  $l =1,2,\ldots,m$ and $\sum_{l=1}\sp{m} \lambda_l = 1$
 (cf., Theorem~\ref{THfavtarProp33}).
Since 
the first component of 
$(a_{i}+a_{k})/2$
is zero,
the first component of each $y\sp{(l)}$ must also be zero,
which means that each
$y\sp{(l)}$ coincides with $a_{j}$ for some $j=j(l) \in I_{1}^{0}$.
Hence we have
$y\sp{(l)}  p \leq f(y\sp{(l)})$
for $l =1,2,\ldots,m$
by (\ref{ineqApbi0}). 
Using this and (\ref{prfFMredun1}) we obtain
\begin{equation*}
 \frac{1}{2}(a_{i} + a_{k}) p
=   
\sum_{l=1}\sp{m} \lambda_l y\sp{(l)}  p 
\leq  
\sum_{l=1}\sp{m} \lambda_l f(y\sp{(l)}) 
\leq  \frac{1}{2} (   f(a_{i}) +f(a_{k}) ) .
\end{equation*}
The above argument shows that (\ref{FMp2pnvar}) can be derived from 
the inequalities in (\ref{ineqApbi0}).
\end{proof}

For $j=2,3,\ldots,n$, define
\begin{align*}
 I_{j}^{+} &= \{ i \in I_{j-1}^{0} \mid a_{ij} = +1 \},   
\\ 
 I_{j}^{0} &= \{ i \in I_{j-1}^{0} \mid a_{ij} = 0 \},   
\\ 
 I_{j}^{-} &= \{ i \in I_{j-1}^{0} \mid a_{ij} = -1 \}.
\end{align*}
Then the original system (\ref{ineqApb}) is equivalent to 
\begin{align}
 \max_{k \in I_{1}^{-}} \left\{ \sum_{j=2}^n a_{kj}p_{j} - b_{k} \right\} 
  \leq & \ p_{1} \leq  
 \min_{i \in I_{1}^{+}} \left\{ b_{i} - \sum_{j=2}^n a_{ij}p_{j}  \right\} ,
\nonumber \\
 \max_{k \in I_{2}^{-}} \left\{ \sum_{j=3}^n a_{kj}p_{j} - b_{k} \right\} 
  \leq & \  p_{2} \leq  
 \min_{i \in I_{2}^{+}} \left\{ b_{i} - \sum_{j=3}^n a_{ij}p_{j}  \right\} ,
\nonumber \\
  & \ \vdots 
\label{FMp12n}
\\
 \max_{k \in I_{n-1}^{-}} \left\{ a_{kn}p_{n} - b_{k} \right\} 
  \leq & \  p_{n-1} \leq  
 \min_{i \in I_{n-1}^{+}} \left\{ b_{i} - a_{in}p_{n}  \right\} ,
\nonumber \\
 \max_{k \in I_{n}^{-}} \left\{ - b_{k} \right\} 
  \leq & \ p_{n} \leq  
 \min_{i \in I_{n}^{+}} \left\{ b_{i} \right\} .
\nonumber 
\end{align}
Note that the expressions above are valid even when some of
the index sets $I_{j}^{+}$ and/or $I_{j}^{-}$ are empty.

Since $\subgR f(x)$ is nonempty,
there exists a real vector $p$ satisfying the inequalities (\ref{FMp12n}).
As for integrality, the last inequality in (\ref{FMp12n}) shows
that we can choose an integral $p_{n} \in \ZZ$,
since $b_{i} = f(a_{i})$ 
for $i \in I_{n}^{-} \cup I_{n}^{+}$
 and $f$ is integer-valued.
Then the next-to-last inequality shows that
we can choose an integral $p_{n-1} \in \ZZ$,
since
$a_{kn}p_{n} - b_{k} \in \ZZ$ for $k \in I_{n-1}^{-}$
and
$b_{i} - a_{in}p_{n} \in \ZZ$ for $i \in I_{n-1}^{+}$.
Continuing in this way we can see the existence of an integer vector $p \in \ZZ\sp{n}$
satisfying (\ref{FMp12n}).  
This shows  $\subgZ f(x) \not= \emptyset$,
completing the proof of Theorem~\ref{THsubgrIC}.

\begin{remark} \rm \label{BoundedSG}
Suppose that $\subgR f(x)$ is a bounded polyhedron for an integrally convex function 
$f: \mathbb{Z}^{n} \to \mathbb{Z} \cup \{ +\infty  \}$ and $x \in \dom f$.
The expression (\ref{FMp12n}) shows that there exists an integral vertex of  $\subgR f(x)$.
Indeed we can choose the (finite) upper bound in (\ref{FMp12n}) for each $p_i$.
It is emphasized, however, that not every vertex is an integral vector.
\finbox
\end{remark}

\section{Concluding Remarks}
\label{SCconclrem}

The established biconjugacy implies that
there is a one-to-one correspondence
between the class $\calF_{\rm IC}$
of integer-valued integrally convex functions 
and the class of their integral conjugates
$\calF\sp{\bullet}_{\rm IC} = \{ f\sp{\bullet}  \mid  f \in \calF_{\rm IC} \}$.
By the conjugacy theorems related to L- and M-convex functions
(see \cite[Fig.~8.1]{Mdcasiam}),
the class $\calF\sp{\bullet}_{\rm IC}$ also contains
separable convex,
{\rm L}-convex,
${\rm L}^{\natural}$-convex,
{\rm M}-convex,  
${\rm M}^{\natural}$-convex,  
${\rm L}^{\natural}_{2}$-convex, and 
${\rm M}^{\natural}_{2}$-convex functions.
A direct characterization of $\calF\sp{\bullet}_{\rm IC}$ is an interesting question
and is left for the future.
It will be also interesting to characterize its subclasses
$\calF\sp{\bullet}_{\rm BS} = \{  f\sp{\bullet} \mid \mbox{$f$: integer-valued BS-convex} \}$
and
$\calF\sp{\bullet}_{\rm UJ} = \{  f\sp{\bullet} \mid \mbox{$f$: integer-valued UJ-convex} \}$.

\

\noindent {\bf Acknowledgement} 
The authors thank  Satoru Fujishige and Hiroshi Hirai for helpful comments.
This work was supported by CREST, JST, 
Grant Number JPMJCR14D2, Japan, and
JSPS KAKENHI Grant Numbers 26280004, 16K00023.









\end{document}